\documentclass[11pt]{article}
\usepackage{subeqnarray,url}
\usepackage{subfig}
\usepackage{graphicx} 
\usepackage{amsmath,amssymb}
\usepackage[margin=1.5in]{geometry}
\usepackage{url}
\usepackage{authblk}
\usepackage{mathtools}
\usepackage{amsthm}
\usepackage{todonotes}

\newcommand{\RR}{\mathbb{R}}
\newcommand{\NN}{\mathbb{N}}

\newcommand{\cZ}{{\cal Z}}

\newcommand{\cC}{{\cal C}}
\newcommand{\cK}{{\cal K}}
\newcommand{\cF}{{\cal F}}

\newcommand{\cP}{{\cal P}}

\newcommand{\mC}{{\mathbb C}}

\newcommand{\mR}{{\mathbb R}}
\usepackage{algorithm}
\usepackage{algorithmic}
\usepackage[algo2e,linesnumbered,ruled]{algorithm2e}

\newcommand{\mU}{{\mathbb U}}
\newcommand{\bG}{{\mathbf G}}

\newcommand{\bs}{{\mathbf s}}

\newcommand{\bF}{{\mathbf F}}

\newcommand{\bff}{{\mathbf f}}

\newcommand{\bw}{{\mathbf w}}

\usepackage{color}

\newcommand{\bC}{{\boldsymbol{\mathcal C}}}

\newcommand{\bz}{{\mathbf z}}

\newtheorem{theorem}{Theorem}
\newtheorem{definition}{Definition}

\newtheorem{remark}{Remark}

\newtheorem{proposition}{Proposition}

\newcommand{\norm}[2]{\ifthenelse{\equal{#2}{}}{\left\|\cdot\right\|_{#1}}{\left\|#2\right\|_{#1}}}
\newcommand{\seminorm}[2]{\ifthenelse{\equal{#2}{}}{\left|\cdot\right|_{#1}}{\left|#2\right|_{#1}}}

\DeclareMathOperator*{\argmin}{argmin}
\DeclareMathOperator*{\Cov}{Cov}
\usepackage[utf8]{inputenc}

\DeclareUnicodeCharacter{202C}{\textvisiblespace}

\begin{document}
\title{Kernel Methods for the Approximation of the Eigenfunctions of the Koopman Operator}

\author[1]{Jonghyeon Lee}
\author[1,2]{Boumediene Hamzi}
\author[3]{Boya Hou}
\author[1]{Houman Owhadi}
\author[4]{Gabriele Santin}
\author[5]{Umesh Vaidya}

\affil[1]{Department of Computing and Mathematical Sciences, Caltech, CA, USA.}
\affil[2]{Alan Turing Institute, London, UK.}
\affil[3]{Carl R. Woese Institute for Genomic Biology, University of Illinois, Urbana-Champaign, IL, USA.}
\affil[4]{Department of Environmental Sciences, Informatics and Statistics, Ca’ Foscari University of Venice, Italy. }
\affil[5]{Department of Mechanical Engineering, Clemson University, USA. }

\maketitle

\begin{abstract}

The Koopman operator provides a linear framework to study nonlinear dynamical systems. Its spectra offer valuable insights into system dynamics, but the operator can exhibit both discrete and continuous spectra, complicating direct computations. In this paper, we introduce a kernel-based method to construct the principal eigenfunctions of the Koopman operator without explicitly computing the operator itself. These principal eigenfunctions are associated with the equilibrium dynamics, and their eigenvalues match those of the linearization of the nonlinear system at the equilibrium point. We exploit the structure of the principal eigenfunctions by decomposing them into linear and nonlinear components. The linear part corresponds to the left eigenvector of the system's linearization at the equilibrium, while the nonlinear part is obtained by solving a partial differential equation (PDE) using kernel methods. Our approach avoids common issues such as spectral pollution and spurious eigenvalues, which can arise in previous methods. We demonstrate the effectiveness of our algorithm through numerical examples.

\end{abstract}
\section{Introduction}

Time series data, which are ubiquitous in various scientific domains, have driven the development of a wide range of statistical and machine learning forecasting methods \cite{kantz97,CASDAGLI1989, yk1, yk2, yk3, yk4, survey_kf_ann,Sindy,jaideep1,nielsen2019practical,abarbanel2012analysis, pillonetto2011new,wang2011predicting,brunton2016discovering,lusch2018deep,callaham2021learning,kaptanoglu2021physicsconstrained,kutz2022parsimony}. 
Dynamical systems theory provides tools to understand the underlying rules governing the evolution of time series data.
Originating with the work by Koopman \cite{koopman1932dynamical}, Koopman operator theory provides a linear lens to study nonlinear dynamical systems. In particular, the Koopman operator maps the nonlinear evolution of finite-dimensional state space to an infinite-dimensional but linear description of functions. Following the seminal work \cite{mezic2005spectral}, there is a surge in interest in using the Koopman operator to study dynamical systems.
As a linear operator, the spectral of the Koopman operator is rich in information and plays a significant role in various analysis and synthesis problems.
For example, the principal eigenfunctions reveal the state space geometry as the joint zero-level curves characterize the stable and unstable manifolds of the equilibrium points \cite{mauroy2016global,umathe2023spectral}. 
In \cite{matavalam2024data}, the principal eigenfunctions are used to identify the stability boundary, thereby determining the domain of attraction of an equilibrium point. Furthermore, the solution to the optimal control problem can also be constructed using the principal eigenfunctions of the Koopman operator, per \cite{vaidya2022spectral,VaidyaKHJ2024}. In \cite{deka2023path}, the authors proposed a path-integral formula for the computation of Koopman principal eigenfunctions. 
Yet in general, the Koopman operator can have both discrete and continuous spectra \cite{mezic2005spectral}, making the computation challenging. In addition, approximating the spectra of an infinite-dimensional operator via a finite-dimensional matrix can suffer from ''spectral pollution'' \cite{colbrook2020foundations}. In this paper, we aim to develop a numerical method to construct the principal eigenfunction of the Koopman operator directly from data that circumvents these difficulties.

Numerical methods such as extended dynamic mode decomposition (EDMD) \cite{williamsEDMD} aim to approximate the infinite-dimensional Koopman operator via its actions on the finite-dimensional space spaced by a set of pre-selected functions. 
Recent research has also explored the combination of kernel methods and Koopman operator theory \cite{klus2020data,peherstorfer2020koopman,klus2015kernel,DAS2020573,hou2023sparse,ishikawa2024koopmanoperatorsintrinsicobservables}.
Building upon the theory of reproducing kernel Hilbert spaces (RKHS) \cite{CuckerandSmale}, kernel-based methods offer considerable advantages in regularization, guaranteed convergence, automatization, and interpretability \cite{chen2021solving, batlle2025error, houman_cgc, bourdais2024codiscovering}, and have notably strengthened the mathematical basis for analyzing dynamical systems \cite{yk1, bhcm11, bhcm1, lyap_bh, BHPhysicaD, hamzi2019kernel, bh2020b, klus2020data, ALEXANDER2020132520, bhks, bh12, bh17, hb17, mmd_kernels_bh, bh_kfs_p5,bh_kfs_p6,hou2024propagating,cole_hopf_poincare,kernel_sos} as well as surrogate modeling \cite{santinhaasdonk19}.
In this paper, we explore the use of kernel methods for the computation of spectral properties of the Koopman operator. The advantage of using kernel methods in this context is that they provide with error estimates that make the method more rigorous.

The existing literature on kernel-based Koopman learning, e.g. \cite{klus2020data,hou2023sparse}, requires constructing a Koopman operator from data first. The spectra are then computed from the learned Koopman operator.  
Yet the spurious eigenvalues \cite{colbrook2024limits} can arise due to numerical artifacts, noise, or overfitting in the computation. These spurious eigenvalues do not represent genuine modes of the underlying dynamics and are misleading when analyzing the system's behavior. To deal with this challenge, we aim to directly learn the spectrum of the Koopman operator without computing the Koopman operator itself by solving a partial differential equation (PDE) in RKHS. Specifically, following \cite{deka2023path}, we associate the principal eigenfunctions of the Koopman operator with an equilibrium point with the corresponding eigenvalues being the same as that of the linearization of the nonlinear system at said equilibrium point. The linear parts of the principal eigenfunction correspond to the eigenvector of local linearization at the origin, while the nonlinear part of the eigenfunction is the solution of PDEs, which can be solved in RKHS. Unlike \cite{bevanda2024koopman} that learns eigenfunctions of the Koopman operator in RKHS, the decomposition into linear and nonlinear parts allows us to incorporate knowledge from the linearization of system dynamics, and exploit the possible nonlinear structure of RKHS.

The remainder of this paper is organized as follows: In Section 2, we provide background on the Koopman operator and its eigenfunctions. In Section 3, we discuss kernel methods and their application to solving the PDE satisfied by the Koopman eigenfunctions. Section 4 presents error estimates for our method, and in Section 5, we showcase numerical experiments.


\section{Preliminaries on the Koopman Operator}

\subsection{Koopman Operator and its Spectrum}
In this section, we provide a brief overview of the spectral theory of the Koopman operator.
We refer the reader to \cite{mezic2020spectrum,mezic2021koopman} for more details. Consider the dynamical system
\begin{align}
    \dot x={f}(x),\label{odesys}
\end{align}
defined on a state space ${\cal Z}\subseteq \mR^p$.  The vector field $\bff$ is assumed to be smooth function. Let $\cF\subseteq \cC^0$ be the function space of observable $\psi: \cZ\to \mC$.
We have following definitions for the Koopman operator and its spectrum. 
\begin{definition}[Koopman Operator] The family of Koopman  operators $\mathbb{U}_t:\cF\to \cF$ corresponding to ~\eqref{odesys} is defined as 
\begin{eqnarray}[\mathbb{U}_t \psi](x)=\psi(s_t(x)). \label{koopman_operator}
\end{eqnarray}
where $s_t(x)$ is the solution of the dynamical system (\ref{odesys}). 
If in addition $\psi$ is continuously differentiable, then $\varphi(x,t):=[\mU_t \psi ](x)$ satisfies a partial differential equation \cite{Lasota} 
\begin{align}
\frac{\partial \varphi}{\partial t}=\frac{\partial \varphi}{\partial x} f=: \cK_f \varphi \label{Koopmanpde}
\end{align}
with the initial condition $\varphi(x,0)=\psi(x)$. The operator $\cK_\bff$ is the infinitesimal generator of $\mU_t$ i.e.,
\begin{eqnarray}
{\cal K}_{f} \psi=\lim_{t\to 0}\frac{(\mathbb{U}_t-I)\psi}{t}. \label{K_generator}
\end{eqnarray}
\end{definition}
It is easy to check that each $\mU_t$ is a linear operator on the space of functions, $\cF$.  
\begin{definition}\label{definition_koopmanspectrum}[Eigenvalues and Eigenfunctions of Koopman] A function $\psi_\lambda(x)$, assumed to be at least $\cC^1$,  is said to be an eigenfunction of the Koopman operator associated with eigenvalue $\lambda$ if
\begin{eqnarray}
[\mU_t \psi_\lambda](x)=e^{\lambda t}\psi_\lambda(x)\label{eig_koopman}.
\end{eqnarray}
Using the Koopman generator, the (\ref{eig_koopman}) can be written as 
\begin{align}
    \frac{\partial \psi_\lambda}{\partial x}{f}=\lambda \psi_\lambda\label{eig_koopmang}.
\end{align}
\end{definition}
The eigenfunctions and eigenvalues of the Koopman operator enjoy the following property \cite{mezic2020spectrum,budivsic2012applied}. 
In this paper, we are interested in approximating the eigenfunctions of the Koopman operator with associated eigenvalues, the same as that of the linearization of the nonlinear system at the equilibrium point. With the hyperbolicity assumption on the equilibrium point of the system (\ref{odesys}), this part of the spectrum of interest to us is  well-defined. In the following discussion, we summarize the results from \cite{mezic2020spectrum} relevant to this paper and justify some of the claims made above on the spectrum of the Koopman operator. 

Equations (\ref{eig_koopman}) and (\ref{eig_koopmang}) provide a general definition of the Koopman spectrum. However, the spectrum can be defined over finite time or over a subset of the state space. The spectrum of interest to us in this paper could be well-defined over the subset of the state space. 
\begin{definition}[Open Eigenfunction \cite{mezic2020spectrum}]\label{definition_openeigenfunction}
Let $\psi_\lambda: \bC\to \mC$, where $\bC\subset \cZ$ is not an invariant set. Let $x\in  \bC$, and
$\tau \in (\tau^-(x),\tau^+(x))= I_x$, a connected open interval such that $\tau (x) \in \bC$ for all  $\tau \in I_x$.
If
\begin{align}[\mU_\tau \psi_\lambda](x) = \psi_\lambda(\bs_\tau(x)) =e^{\lambda \tau}  \psi_\lambda (x),\;\;\;\;\forall \tau \in I_x. 
\end{align}
Then $\psi_\lambda(x)$ is called the open eigenfunction of the Koopman operator family $\mU_t$, for $t\in \mR$ with eigenvalue $\lambda$. 
\end{definition}
If $\bC$ is a proper invariant subset of $\cZ$ in which case $I_x=\mR$ for every $x\in \bC$, then $\psi_\lambda$ is called the subdomain eigenfunction. If $\bC=\cZ$ then $\psi_\lambda$ will be the ordinary eigenfunction associated with eigenvalue $\lambda$ as defined in (\ref{eig_koopman}). The open eigenfunctions as defined above can be extended from $\bC$ to a larger reachable set when $\bC$ is open based on the construction procedure outlined in  \cite[Definition 5.2, Lemma 5.1]{mezic2020spectrum}. Let $\cP$ be that larger domain. The eigenvalues of the linearization of the system dynamics at the origin, i.e., $E$, will form the eigenvalues of the Koopman operator \cite[Proposition 5.8]{mezic2020spectrum}. Our interest will be in constructing the corresponding eigenfunctions, defined over the domain $\cP$. We will refer to these eigenfunctions as {\it principal eigenfunctions} \cite{mezic2020spectrum}. \\
The principal eigenfunctions can be used as a change of coordinates in the linear representation of a nonlinear system and draw a connection to the famous Hartman-Grobman theorem  on linearization and Poincare normal form \cite{arnold2012geometrical}. 
The principal eigenfunctions will be defined over a proper subset $\cP$ of the state space $\cZ$ (called subdomain eigenfunctions) or over the entire $\cZ$ \cite[Lemma 5.1, Corollary 5.1, 5.2, and 5.8]{mezic2020spectrum}. \\

\subsection{Decomposition of Koopman Eigenfunctions}



Following \cite{deka2023path}, we decompose principle eigenfunctions into linear and nonlinear parts. Specifically, consider the decomposition of the nonlinear system into linear and nonlinear parts as
 \begin{equation}
 \dot x=f(x)=Ex+(f(x)-E x)=:E x+ G(x)\label{sys_decompose}.
 \end{equation}
where $E=\frac{\partial f}{\partial x}(0)$ with $Ex$ the linear part and $G$ the purely nonlinear part. 
For the simplicity of presentation and continuity of notations, we present approximation results for eigenfunctions with simple real eigenvalues; the extension to the complex case is deferred to future work. Let $\lambda$ be the eigenvalues of the Koopman generator and also of $E$. The eigenfunction corresponding to eigenvalue $\lambda$ admits the decomposition into linear and nonlinear parts.
\begin{eqnarray}
 \phi_\lambda(x)=w^\top x+h(x) \label{eigen_decompose},
\end{eqnarray}
where $ w^\top x$ and $h(x)$ are the eigenfunction's linear and purely nonlinear parts, respectively. Substituting (\ref{eigen_decompose}) in following general expression of Koopman eigenfunction i.e.,
\begin{eqnarray}
 \frac{\partial \phi_\lambda}{\partial x}(x)
 \cdot
 f(x)=\lambda \phi_{\lambda}(x)
\end{eqnarray}
and using (\ref{sys_decompose}), we obtain following equations to be satisfied by $w$ and $h(x)$
\begin{eqnarray}
w^\top E=\lambda  w^\top,\;\;\frac{\partial h}{\partial x}(x) \cdot f(x)-\lambda h(x)+w^\top G(x)=0
\label{linear_nonlinear_eig}.
 \end{eqnarray}
So, the linear part of the eigenfunction can be found as the left eigenvector with eigenvalue $\lambda$ of the matrix $E$, and the nonlinear term satisfies the linear partial differential equation.
\section{RKHS-based Computational Framework}

\subsection{Reproducing Kernel Hilbert Spaces (RKHS)}

We give a brief overview of reproducing kernel Hilbert spaces as used in statistical learning
theory ~\cite{CuckerandSmale}. Early work developing
the theory of RKHS was undertaken by N. Aronszajn~\cite{aronszajn50reproducing}.

\begin{definition} Let  ${\mathcal H}$  be a Hilbert space of functions on a set ${\mathcal X}$.
Denote by $\langle f, g \rangle$ the inner product on ${\mathcal H}$   and let $\|f\|= \langle f, f \rangle^{1/2}$
be the norm in ${\mathcal H}$, for $f$ and $g \in {\mathcal H}$. We say that ${\mathcal H}$ is a reproducing kernel
Hilbert space (RKHS) if there exists a function $k:{\mathcal X} \times {\mathcal X} \rightarrow \RR$
such that\\
\begin{itemize}
 \item[i.] $k_x:=k(x,\cdot)\in{\mathcal{H}}$ for all $x\in{\mathcal{H}}$.
\item[ii.] $k$ spans ${\mathcal H}$: ${\mathcal H}=\overline{\mbox{span}\{k_x~|~x \in {\mathcal X}\}}$.
 \item[iii.] $k$ has the {\em reproducing property}:
$\forall f \in {\mathcal H}$, $f(x)=\langle f,k_x \rangle$.
\end{itemize}
$k$ will be called a reproducing kernel of ${\mathcal H}$. ${\mathcal H}_k$  will denote the RKHS ${\mathcal H}$
with reproducing kernel $k$ where it is convenient to explicitly note this dependence.
\end{definition}

The important properties of reproducing kernels are summarized in the following proposition.
\begin{proposition}\label{prop1} \cite{aronszajn50reproducing}
If $k$ is a reproducing kernel of a Hilbert space ${\mathcal H}$, then\\
\begin{itemize}
\item[i.] $k(x,y)$ is unique.
\item[ii.]  $\forall x,y \in {\mathcal X}$, $k(x,y)=k(y,x)$ (symmetry).
\item[iii.] $\sum_{i,j=1}^q\alpha_i\alpha_j k(x_i,x_j) \ge 0$ for $\alpha_i \in \RR$, $x_i \in {\mathcal X}$ and $q\in\mathbb{N}_+$
(positive definiteness).
\item[iv.] $\langle k(x,\cdot),k(y,\cdot) \rangle=K(x,y)$.
\end{itemize}
\end{proposition}
Common examples of reproducing kernels defined on a compact domain $\mathcal{X} \subset \mathrm{R}^n$ are the 
(1) constant kernel: $K(x,y)= m > 0$
(2) linear kernel: $k(x,y)=x\cdot y$
(3) polynomial kernel: $k(x,y)=(1+x\cdot y)^d$ for $d \in \NN_+$
(4) Laplace kernel: $k(x,y)=e^{-||x-y||_2/\sigma^2}$, with $\sigma >0$
(5)  Gaussian kernel: $k(x,y)=e^{-||x-y||^2_2/\sigma^2}$, with $\sigma >0$
(6) triangular kernel: $k(x,y)=\max \{0,1-\frac{||x-y||_2^2}{\sigma} \}$, with $\sigma >0$.
(7) locally periodic kernel: $k(x,y)=\sigma^2 e^{-2 \frac{ \sin^2(\pi ||x-y||_2/p)}{\ell^2}}e^{-\frac{||x-y||_2^2}{2 \ell^2}}$, with $\sigma, \ell, p >0$.

\begin{theorem} \label{thm1} \cite{aronszajn50reproducing}
Let $k:{\mathcal X} \times {\mathcal X} \rightarrow \RR$ be a symmetric and positive definite function. Then there
exists a Hilbert space of functions ${\mathcal H}$ defined on ${\mathcal X}$   admitting $k$ as a reproducing Kernel.
Conversely, let  ${\mathcal H}$ be a Hilbert space of functions $f: {\mathcal X} \rightarrow \RR$ satisfying
$\forall x \in {\mathcal X}, \exists \kappa_x>0,$ such that $|f(x)| \le \kappa_x \|f\|_{\mathcal H},
\quad \forall f \in {\mathcal H}. $
Then ${\mathcal H}$ has a reproducing kernel $k$.
\end{theorem}


\begin{theorem}\label{thm4} \cite{aronszajn50reproducing}
 Let $k(x,y)$ be a positive definite kernel on a compact domain or a manifold $X$. Then there exists a Hilbert
space $\mathcal{F}$  and a function $\Phi: X \rightarrow \mathcal{F}$ such that
$$k(x,y)= \langle \Phi(x), \Phi(y) \rangle_{\mathcal{F}} \quad \mbox{for} \quad x,y \in X.$$
 $\Phi$ is called a feature map, and $\mathcal{F}$ a feature space\footnote{The dimension of the feature space can be infinite, for example in the case of the Gaussian kernel.}.
\end{theorem}


Theorem \ref{thm4} is often referred to as the ``kernel trick'', and its utility lies in the fact that the kernel function obviates the need to compute high-dimensional outputs of the feature map $\Phi$ directly.

\begin{theorem}\label{thm5} \cite{owhadi_scovel_2019}
Let $k: X \times X \rightarrow \mathbb{R}$ be a real-valued kernel and $K$ be the associated RKHS of functions mapping $X$ to $\mathbb{R}$. Let $\Phi = (\Phi_1,\dots,\Phi_n) \in (K^*)^n$ be a vector of linear functionals from $K$ to $\mathbb{R}$, and write $\Phi (u) = ([\Phi_1,u],\dots,[\Phi_n,u]) \in \mathbb{R}^n$, where $[\cdot,\cdot]$ is the dual pairing. Then for $y \in \mathbb{R}^n$, 

\begin{equation}
\argmin_{u \in \mathcal{H}_K \textrm{s.t.} \Phi(u) = y} ||u||_K = \mathbb{E}[\xi|\Phi(\xi)=y].
\end{equation}

Here, $\xi$ is a centered Gaussian process with covariance $K: \mathcal{B} \rightarrow \mathcal{B}^*$, where $\mathcal{B}$ is a separable Banach space, which is a linear bijection that is symmetric ($[\phi,K \varphi]=[\varphi, K \phi]$) and positive ($[\phi,K\phi]>0$ for $\phi \neq 0$). 
\end{theorem}

Theorem \ref{thm5} is significant because it states that the problem of recovering a sufficiently regular function $u$ with respect to constraints $\Phi(u) = y$ is equivalent to finding the conditional expectation of $\xi$, the Gaussian process approximation of $u$, given observations $\Phi(\xi) = y$. Fortunately, (\ref{thm5}) has a closed-form solution given in the theorem below:

\begin{theorem}\label{thm6} \cite{owhadi_scovel_2019}
The conditional expectation of a Gaussian process $\xi$ in $\mathcal{B}$ given observations $\Phi (u)$ is given by the following representer formula:

\begin{equation}
    \mathbb{E}[\xi|\Phi(u)=y] = \sum^m_{i,j = 1} [\Phi_i, \xi] \Theta^{-1}_{i,j} T \Phi_j \in \mathcal{B}
\end{equation}
\end{theorem}

We outline the key steps in proving Theorem (\ref{thm6}) and refer interested readers to \cite{owhadi_scovel_2019} for details. 
Consider the conditional expectation of $[\xi,\psi]$ given $\Phi(u)$, where $\psi$ is an arbitrary element of the dual Banach space $\mathcal{B}^*$. Since $[\xi,\psi] \in \mathbb{R}$ is a real-valued Gaussian variable, we may use results from classical probability theory to show that $\mathbb{E}[[\psi, \xi] |\Phi (u)] = \sum^m_{i=1} c_i [\Phi_i,\xi]$, where $c_i$ are coefficients that satisfy $[\psi,\xi] - \sum^m_{i=1} c_i [\Phi_i,\xi]$ are independent from $[\Phi_j,\xi]$ $\forall j$. This in turn implies that $\Cov[[\psi,\xi]-\sum^m_{i=1} c_i [\Phi_i,\xi],[\Phi_j,\xi]]=0 \Leftrightarrow [\psi,T\Phi_j]- \sum^m_{i=1} c_i [\Phi_i, T\Phi_j] = 0$. Now define $\Theta$ to be an $m \times m$ matrix with entries $\Theta_{i,j} = [\Phi_i,T\Phi_j]$, which is equivalent to stating that $[\psi,T\Phi_j] = \sum^m_{i=1} c_i \Theta_{i,j}$. If $\Theta$ is invertible, $c_i = \sum^m_{j=1} \Theta^{-1}_{i,j} [\psi,T \Phi_j]$. Therefore, we have
\begin{align}
\mathbb{E}[[\psi,\xi]|[\Phi_1,\xi],\dots,[\Phi_m,\xi]] 
= \sum^m_{i,j=1} [\Phi_i,\xi] \Theta^{-1}_{i,j} \underbrace{T \Phi_j}_{\in \mathcal{B}} 
= [\psi,\sum^m_{i,j=1} [\Phi_i,\xi] \Theta^{-1}_{i,j} T \Phi_j].
\end{align}
Since this is true for all $\psi \in \mathcal{B}^*$, we conclude that
\begin{equation}\label{generalized_representer}
    \mathbb{E}[\xi|[\Phi_1,\xi],\dots,[\Phi_m,\xi]] = \sum^m_{i,j=1} [\Phi_i,\xi] \Theta^{-1}_{i,j} T\Phi_j \in \mathcal{B}.
\end{equation}

In the next section, we write the generalized representer formula (\ref{generalized_representer}) in the equivalent form

\begin{equation}
u (\cdot) =  K(\cdot,\tilde{\phi})K(\tilde{\phi},\tilde{\phi})^{-1} Y, 
\end{equation}

where $K(\tilde{\phi},\tilde{\phi}) = \Theta$, $Y_i = [\tilde{\phi}_i,\xi]$ and $K(\cdot,\tilde{\phi}_i) = T\tilde{\phi}_i $

\subsection{Solving the linear PDE from data}

Solving the linear PDE \eqref{linear_nonlinear_eig} can be framed as a quadratic optimization problem which can be solved by kernel regression. For each eigenvalue $\lambda_k, \ k=1,\dots,n$ , define the best Gaussian approximation of $h_{\lambda_k}$, where $h_{\lambda_k}$ is a function of $d$ variables, to be the $h^*_{\lambda_k}$ that satisfies

\begin{equation}
\begin{aligned}
\min_{h^*_{\lambda_k} \in \mathcal{H}_K} \quad & ||h^*_{\lambda_k}||^2_K \\
\textrm{s.t.} \quad & h^*_{\lambda_k} (\textbf{0})  = 0 \\
\quad & \frac{\partial}{\partial z_j} h^*_{\lambda_k} (\textbf{0}) = 0, \; j = 1, \dots, d \\
\quad &  \frac{\partial h^*_{\lambda_k}}{\partial \bz}(\bz_i) \cdot \bF(\bz_i) -\lambda_k h^*_{\lambda_k} (\bz_i)= -\bw^T_k \bG(\bz_i),  \; i = 1, \dots, N\\
\end{aligned}
\label{optproblem}
\end{equation}
where $\bz_i : = (z_{1,i},z_{2,i},\dots,z_{d,i})$ denotes our collocation points.

Define
\begin{align}
\begin{aligned}
\tilde{\phi}_1 (\bz)=& \delta_{\textbf{0}} (\bz), \\
\tilde{\phi}_{1+i} (\bz) = & \delta_{\textbf{0}} (\bz) \cdot \frac{\partial}{\partial z_i}, \; i=1,\dots,d\\
\tilde{\phi}_{1+d+i} (\bz) =& \sum^{d}_{j=1} F(\bz_i)_j \delta_{\bz_i} (\bz)\cdot  \frac{\partial}{\partial z_j} - \lambda_k \delta_{\bz_i}(\bz) , \; i=1,\dots,N,
\end{aligned},
\end{align}
where $\delta_{\bz_i}(\bz)$ is the Dirac delta distribution centered at $\bz_i$.

Then the optimization problem \eqref{optproblem} has an explicit solution given by the representer formula

\begin{equation}\label{eq:h_represenation}
    h^*_k (\bz) = K(\bz,\tilde{\phi})(K(\tilde{\phi},\tilde{\phi})+\eta I)^{-1} Y,
\end{equation}
where $K(\bz,\tilde{\phi})$ is a vector of length $N+d+1$ with entries 
\begin{align}
K(\bz,\tilde{\phi})_i := [K(\bz,\cdot),\phi_i]
=\int_{\mathbb{R}^{d}} K(\bz,\bz') \tilde{\phi}_i (\bz') d\bz',   
\end{align}
and $K(\tilde{\phi},\tilde{\phi}) $ is a $(N+d+1)\times (N+d+1)$ matrix with entries
\begin{align}
K(\tilde{\phi}_i,\tilde{\phi}_j) := \int_{\mathbb{R}^{2d}} \tilde{\phi}_i(\bz)K(\bz,\bz') \tilde{\phi}_j(\bz') d\bz d\bz', 
\end{align}
and $Y \in \mathbb{R}^{N+d+1}$ is a vector with entries $Y_1 =\dots = Y_{d+1}= 0$ and $Y_{d+1+i}  = -\bG(\bz_i)^T\bw_k, \; i = 1,\dots,N$. $\eta$ is a small positive regularization constant that reduces numerical errors associated with inverting the matrix $K(\tilde{\phi},\tilde{\phi})$.

\section{Error Estimates}
\begin{theorem}[Validity of Stability and Smoothness Assumptions]\label{theorem_stability_validity}
    Let the equilibrium point \( x_e \) of the ODE \eqref{odesys} be hyperbolic. Assume that $f\in C^m(\Omega, \RR^d)$ for some integer \( m \geq 1 \), where \( \Omega \) is an open neighborhood around \( x_e \).
    
    Let \(\mathbf{w} \in \mathbb{R}^d\) be a fixed weight vector. Consider the partial differential equation (PDE)
    \begin{equation}
    D h(x) := \nabla h(x) \cdot f(x) - \lambda h(x) = -\mathbf{w}^\top G(x), \quad x \in \Omega, \label{linear_nonlinear_eig2}
    \end{equation}
    with boundary conditions
    \begin{equation}
    h(x_e) = 0, \quad \nabla h(x_e) = 0. \label{boundary_conditions}
    \end{equation}

    Then, the solution \( h(x) \) satisfies:
    \begin{enumerate}
        \item \( h(x) \in W^{m+1}_2(\Omega) \), i.e., \( h(x) \) belongs to the Sobolev space \( W^{m+1}_2(\Omega) \).
        \item There exists a stability bound of the form
        \begin{equation}
        \| h \|_{L_p(\Omega)} \leq C_D \| D h \|_{L_q(\Omega)} + C_0 |h(x_e)| + C^{\prime}_0 \| \nabla h(x_e) \|_{\ell^r}, \label{stability_bound}
        \end{equation}
                where \( C_D, C_0, C^{\prime}_0 > 0 \) are constants, and \( p, q, r \in [1, \infty] \).
       
    \end{enumerate}
\end{theorem}

\begin{proof}
We first observe that the nonlinear term \( G(x) = f(x) - E x \), where \( E = \left.\frac{\partial f}{\partial x}\right|_{x_e} \), has smoothness \( G(x) \in \mathcal{C}^m(\Omega, \mathbb{R}^d) \).

\textbf{Step 1: Injectivity of \( D \)}

Since \( x_e \) is hyperbolic, the Jacobian matrix \( E \) has no eigenvalues with zero real parts. This implies that the linearized system around \( x_e \) has no neutral modes, and thus the operator \( D \) associated with the PDE has no non-trivial solutions to the homogeneous equation \( D h = 0 \).

By the Hartman-Grobman theorem, the behavior of the nonlinear system near \( x_e \) is qualitatively similar to its linearization. Therefore, the vector field \( f(x) \) can be approximated by \( f(x) \approx E(x - x_e) \) near \( x_e \), and the operator \( D \) reduces to a linear problem in this neighborhood. The only solution to \( D h = 0 \) near \( x_e \) is the trivial solution \( h(x) = 0 \).

Thus, \( D \) is injective: 

\[
\ker(D) = \{ h \in W^{m+1}_2(\Omega) \mid D h = 0 \} = \{0\}.
\]

    \textbf{Step 2: Surjectivity of \( D \)}

To prove surjectivity, consider the equation:

\[
D h(x) = - \mathbf{w}^\top G(x),
\]
where \( G(x) \in \mathcal{C}^m(\Omega) \). The solution to this equation is governed by  regularity theory of Elliptic PDEs. Specifically, for smooth forcing terms \( G(x) \), elliptic regularity guarantees that the operator \( D \) will have a solution in the Sobolev space \( W^{m+1}_2(\Omega) \), i.e., \( h(x) \in W^{m+1}_2(\Omega) \). 

Elliptic regularity theory ensures that if the right-hand side \( -\mathbf{w}^\top G(x) \) is sufficiently smooth, the solution \( h(x) \) will also be smooth and lie in \( W^{m+1}_2(\Omega) \). Therefore, \( D \) is surjective, meaning that for any smooth \( G(x) \), there exists a solution \( h(x) \in W^{m+1}_2(\Omega) \).

    \textbf{Step 3: Bounded Invertibility via the Banach Inverse Mapping Theorem}

    Since \( D \) is a bounded linear operator (under the smoothness and boundedness assumptions on \( f(x) \) and \( \lambda \)) that is both injective and surjective, it is a bounded linear bijection between \( W^{m+1}_2(\Omega) \) and \( L_q(\Omega) \).

    By the Banach Inverse Mapping Theorem, the inverse operator \( D^{-1} \) exists and is bounded. Therefore, there exists a constant \( C_D > 0 \) such that:
    \[
    \| h \|_{W^{m+1}_2(\Omega)} \leq C_D \| D h \|_{L_q(\Omega)}
    \]

   \textbf{Step 4: Stability Bound for \( h(x) \)}

The solution \( h(x) \) lies in the Sobolev space \( W^{m+1}_2(\Omega) \), and we apply Sobolev embedding theorems to obtain a bound for \( h(x) \) in \( L_p(\Omega) \). Sobolev embedding states that for an embedding \( W^{m+1}_2(\Omega) \subset L_p(\Omega) \), the solution \( h(x) \) will be in \( L_p(\Omega) \) for some \( p \), provided that \( m \) is large enough. 

The stability bound is expressed as:

\[
\| h \|_{L_p(\Omega)} \leq C_D \| D h \|_{L_q(\Omega)} + C_0 |h(x_e)| + C^{\prime}_0 \| \nabla h(x_e) \|_{\ell^r}.
\]

This bound reflects the control over \( h(x) \) both in terms of the forcing term \( G(x) \) (through the \( L_q \)-norm of \( D h(x) \)) and the boundary conditions \( h(x_e) \) and \( \nabla h(x_e) \).

\end{proof}

To formulate the following convergence result we make use of the fill distance $\rho_{Z,\Omega}$ of the finite set $Z\subset\Omega$, given by
\begin{equation*}
    \rho_{Z,\Omega}:=\sup_{x\in\Omega}\min_{z\in Z}\|x-z\|_2.
\end{equation*}
\begin{theorem}\label{th:error}
Assume that Theorem~\ref{theorem_stability_validity} holds with constants $m>d/2$, $p,q\in[1,\infty]$, $C_D>0$, and let $h_{\lambda}$ be the solution of the PDE~\eqref{linear_nonlinear_eig2}.
Assume furthermore that $\mathcal H\hookrightarrow W_2^{m+1}(\Omega)$, and let $h_\lambda^*$ be the solution of the corresponding optimization problem~\eqref{optproblem} with collocation points $Z\subset \Omega$ and with no regularization (i.e., $\eta=0$ in~\eqref{eq:h_represenation}).

Then there are constants $\rho,C>0$ depending on $d, m, \Omega, p, q$, but not on $\lambda, f, Z, h_\lambda$, such that if $\rho_{Z,\Omega}< \rho_0$ it holds
\begin{equation*}
\norm{L_p(\Omega)}{h_\lambda - h_\lambda^*} 
\leq C \left(\norm{W_2^m(\Omega,\RR^d)}{f} + |\lambda|\right) \rho_{Z,\Omega}^{m - d\left(\frac12 -\frac1q\right)_+}\norm{K}{h_\lambda},
\end{equation*}
where $(x)_+\coloneqq \max(x, 0)$.
\end{theorem}
\begin{proof}
We first show that $h\in W_2^{m+1}(\Omega)$ implies that $Dh\in W_2^m(\Omega)$, and give an estimate on its norm.
Namely, for any $\alpha\in\NN_0^d$ with $|\alpha|\leq m$ we have
\begin{align*}
\partial^{\alpha}D h(z)
&=\partial^{\alpha} \left(f(z)^T \nabla h(z) - \lambda h(z)\right)\\
&=\sum_{i=1}^d \left(\partial^{\alpha} f_i(z)\right) (\nabla h(z))_i + \sum_{i=1}^d f_i(z) \partial^\alpha(\nabla h(z))_i- \lambda \partial^{\alpha} h(z),
\end{align*}
and thus there is a constant $C_d$ depending on $d$ such that
\begin{align*}
\norm{L_2}{\partial^{\alpha}D h(z)}
&\leq d \max_{1\leq i\leq d}\left(\norm{L_2}{\partial^{\alpha} f_i} \norm{L_2}{(\nabla h)_i} + \norm{L_2}{f_i}  \norm{L_2}{\partial^\alpha(\nabla 
h)_i}\right) + |\lambda| \norm{L_2}{\partial^{\alpha} h}\\
&\leq C_d \left(\norm{W_2^m(\Omega,\RR^d)}{f} + |\lambda|\right)\norm{W_2^{m+1}}{h}.
\end{align*}
This in turns implies that for some constant $C_{d, m}>0$ depending on $d$ and $m$ we have 
\begin{equation}\label{eq:norm_dh_vs_norm_h}
\norm{W_2^m}{D h} 
=\bigg(\sum_{|\alpha|\leq m}\lVert \partial^{\alpha}D h\lVert^2_{L_2}\bigg)^{1/2}
\leq C_{d,m} \left(\norm{W_2^m(\Omega,\RR^d)}{f} + |\lambda|\right)\norm{W_2^{m+1}}{h}.
\end{equation}
This proves in particular that $D h^*_{\lambda}\in W_2^m(\Omega)$ since $h^*_\lambda\in\mathcal H\hookrightarrow W_2^{m+1}$, while $D h_\lambda\in W_2^{m+1}(\Omega)$ by assumption and by Theorem~\ref{theorem_stability_validity}.

Since $m>d/2$, we may now use the zero lemma of Theorem 12 in~\cite{Narcowich2004}, stating that there are $\rho_0, C>0$, depending on $\Omega, p, q$, such that if $\rho_{Z,\Omega}<\rho_0$ it holds
\begin{equation}\label{eq:sampling}
\norm{L_q(\Omega)}{D h_\lambda - D h_\lambda^*}
\leq C \rho_{Z,\Omega}^{m - d\left(\frac12 -\frac1q\right)_+}\norm{W_2^m(\Omega)}{D h_\lambda - D h_\lambda^*},
\end{equation}
and using the linearity of $D$ and the bound~\eqref{eq:norm_dh_vs_norm_h} this gives
\begin{equation}\label{eq:intermediate_bound}
\norm{L_q}{D h_\lambda - D h_\lambda^*}
\leq C C_{d,m} \left(\norm{W_2^m(\Omega,\RR^d)}{f} + |\lambda|\right)\rho_{Z,\Omega}^{m - d\left(\frac12 -\frac1q\right)_+}\norm{W_2^{m+1}}{h_\lambda - h_\lambda^*}.
\end{equation}
We conclude by bounding the norm in the right-hand side. Since $h_\lambda$ satisfies the constraints of the problem~\eqref{optproblem}, and $h^*_\lambda$ is its minimal norm solution, we have $\norm{K}{h^*_\lambda}\leq \norm{K}{h_\lambda}$. Furthermore, by assumption there is $C_e>0$ such that $\norm{W_2^{m+1}}{h}\leq C_e\norm{K}{h}$ for all $h\in \mathcal H$. It follows that 
\begin{equation*}
\norm{W_2^{m+1}}{h_\lambda - h^*_\lambda}
\leq C_e\norm{K}{h_\lambda - h^*_\lambda}
\leq C_e\left(\norm{K}{h_\lambda} + \norm{K}{h^*_\lambda}\right)
\leq 2 C_e \norm{K}{h_\lambda}, 
\end{equation*}
and thus~\eqref{eq:intermediate_bound} simplifies to
\begin{equation*}
\norm{L_q}{D h_\lambda - D h_\lambda^*}
\leq 2 C C_{d, m} C_e \left(\norm{W_2^m(\Omega,\RR^d)}{f} + |\lambda|\right)\rho_{Z,\Omega}^{m - d\left(\frac12 -\frac1q\right)_+}\norm{K}{h_\lambda}.
\end{equation*}
Inserting this inequality in the stability bound~\eqref{stability_bound} gives the result with an appropriate constant $C>0$, since both $h_\lambda$ and $h^*_\lambda$ satisfy the boundary conditions~\eqref{boundary_conditions}.
\end{proof}
\begin{remark}
We point out that in Theorem~\ref{th:error} the factor $(1/2-1/q)_+$ is decreasing in $q$ and vanishes for $1\leq q\leq 2$. It follows that, given a target $p$, one should choose the smallest $q$ such that~\eqref{stability_bound} holds.  
\end{remark}

\section{Examples}







\subsection{First Analytical Example}

Consider the following dynamical system with an equilibrium point at the origin.

\begin{equation}
\dot x = 
 \left[ \begin{array}{c}
 -2\lambda_2 x_2(x^2_1 - x_2 - 2x_1 x^2_2 + x^4_2 ) + \lambda_1 (x_1 + 4 x^2_1 x_2 - x^2_2 - 8x_1 x^3_2 + 4x^5_2 )
 \\  2\lambda_1 (x_1 - x^2_2 )^2 - \lambda_2(x^2_1 -x_2 - 2x_1 x^2_2 + x^4_2 )
\end{array}\right].
\end{equation}

The eigenvalues of the linearization of the system at the origin i.e, $E$ are $\lambda_1 = -1$ and $\lambda_2 = 3$. For this example, the principal eigenfunctions can be computed analytically and are as follows:

\begin{gather*}
\phi_{\lambda_1}(x)= x_1-x^2_2,\quad  \lambda_1 = -1 \text{ and}\\
\phi_{\lambda_2}(x)= -x_1^2+x_2+2x_1x^2_2-x^4_2,\quad \lambda_2 = 3.
\end{gather*}
Using 3600 points over the square $[-1,1] \times [-1,1]$, we apply kernel regression with the 2D Gaussian kernel 

\begin{equation}
    K(x_1,x_2;\sigma,\sigma_2) = \exp \Big(-\frac{(x_1-y_1)^2}{2\sigma^2_1}-\frac{(x_2-y_2)^2}{2\sigma^2_2}\Big)
\end{equation}
to learn the eigenfunction $\phi_{\lambda_k}$ and depict the results in (\ref{plot1a}) and (\ref{plot1b}). For $\phi_{\lambda_1}$, we take $\sigma_1=\sigma_2=2$; for $\phi_{\lambda_2}$, we use $\sigma_1=2$ and $\sigma_2=3$. The pointwise relative error is very low with the exception of the points where the true solution is equal to zero.

\begin{figure}[tbh]
\centering
\subfloat{\label{fig:test1}
 \includegraphics[width=.31\textwidth]{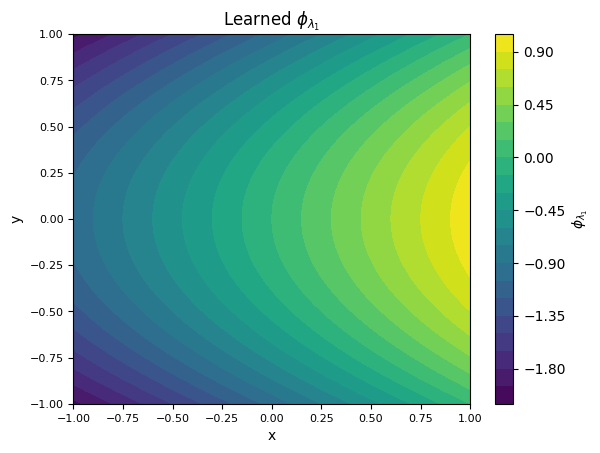}
  }
\hfill
\subfloat{\label{fig:test2}
\includegraphics[width=.31\textwidth]{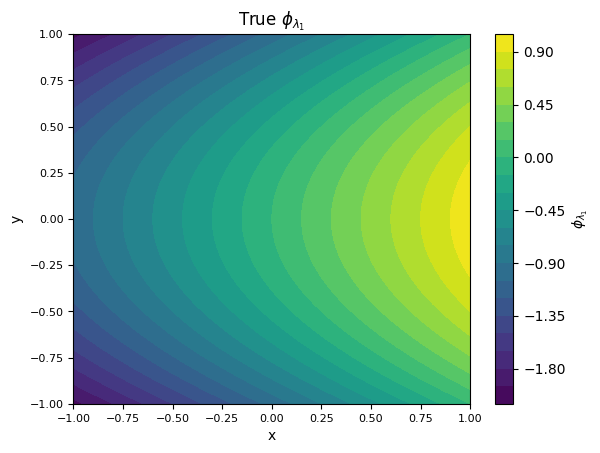}
}
\hfill
\subfloat{\label{fig:error1}
\includegraphics[width=.31\textwidth]{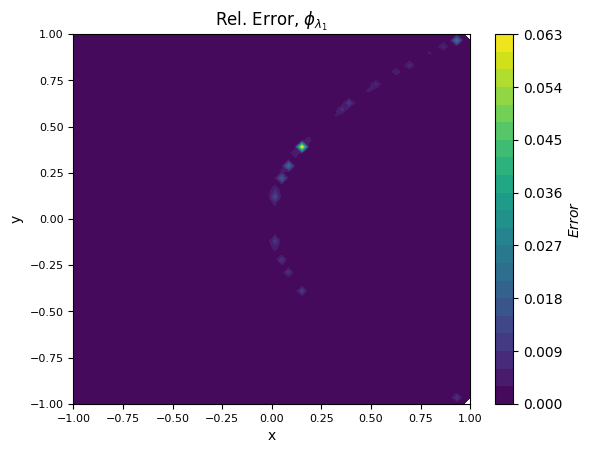}
}
\hfill
\caption{Learned $\phi^*_{\lambda_1}$ (left), true $\phi_{\lambda_1}$ (center) and relative error (right)}
\label{plot1a}
\end{figure}

\begin{figure}[tbh]
\centering
\subfloat{\label{fig:test3}
\includegraphics[width=.31\textwidth]{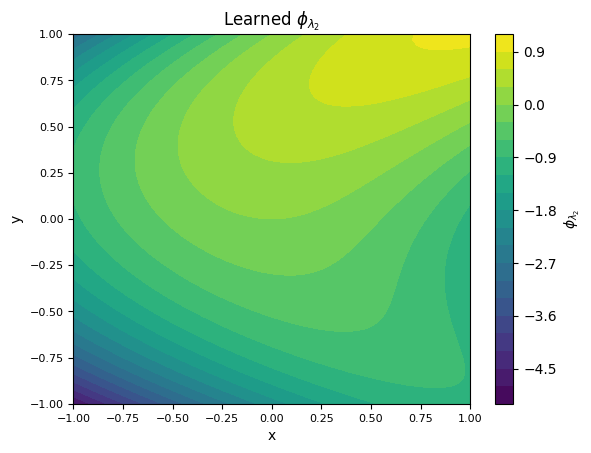}
}
\hfill
\subfloat{\label{fig:test4}
\includegraphics[width=.31\textwidth]{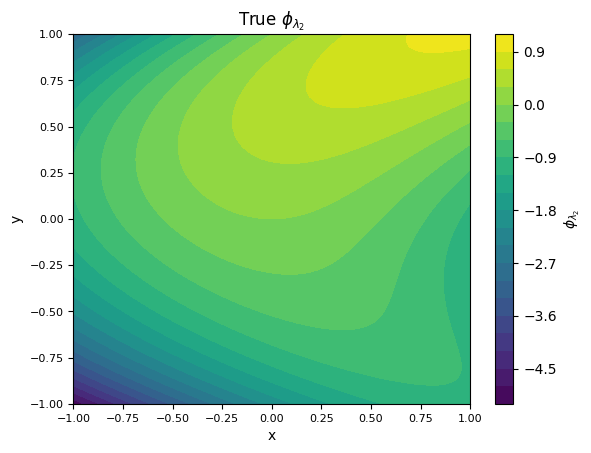}
}
\hfill
\subfloat{\label{fig:error2}
\includegraphics[width=.31\textwidth]{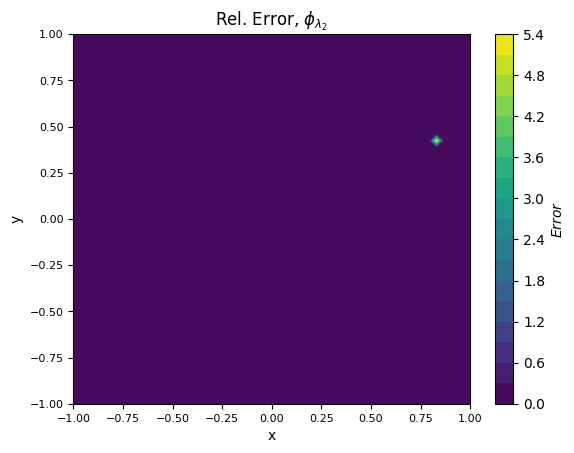}
}
\caption{Learned $\phi^*_{\lambda_2}$ (left), true $\phi_{\lambda_2}$ (center) and relative error (right)}
\label{plot1b}
\end{figure}

\subsection{Second Analytical Example}
\begin{equation}
\dot{x} = 
 \left[ \begin{array}{c}
 \frac{\left(7.5 x_2^2+5.0\right)\left(x_1^3+x_1+\sin \left(x_2\right)\right)+\left(-x_1+x_2^3+2 x_2\right) \cos \left(x_2\right)}{9 x_1^2 x_2^2+6 x_1^2+3 x_2^2+\cos \left(x_2\right)+2} \\
\\
\frac{2.5 x_1^3+2.5 x_1-\left(3 x_1^2+1\right)\left(-x_1+x_2^3+2 x_2\right)+2.5 \sin \left(x_2\right)}{9 x_1^2 x_2^2+6 x_1^2+3 x_2^2+\cos \left(x_2\right)+2}
\end{array}\right].
\end{equation}
The vector field for this system is smooth everywhere on $\mathbb{R}^2$, with a saddle point at the origin. One can verify that the two principal eigenpairs of this system are given by
\begin{gather*}
\phi_{\lambda_1}(x)= x_1 - 2x_2 - x_2^3,\quad  \lambda_1 = -1 \text{ and}\\
\phi_{\lambda_2}(x)= x_1 + \sin(x_2) + x_1^3,\quad \lambda_2 = 2.5.
\end{gather*}

We apply the method outlined in Section 3.2 and use the 2D Gaussian kernel to solve the PDE for each eigenvalue over 2500 points in the grid $[1.5,2.5]\times[1.5,2.5]$,

For $\lambda_1$, we set $\sigma_1=\sigma_2=3$ and for $\lambda_2$, we use $\sigma_1=\sigma_2=7$. As shown by Figures \ref{linplot1a} and \ref{linplot2a}, the solution $\phi^*_{\lambda_k}$, where the nonlinear part $h^*_{\lambda_l}$ has been recovered by the Gaussian kernel, is an accurate approximation of the true eigenfunction $\phi_{\lambda_k}$ for both eigenvalues $\lambda_k$.

\begin{figure}[tbh]
\centering
\subfloat{\label{fig:bigtest1}
 \includegraphics[width=.31\textwidth]{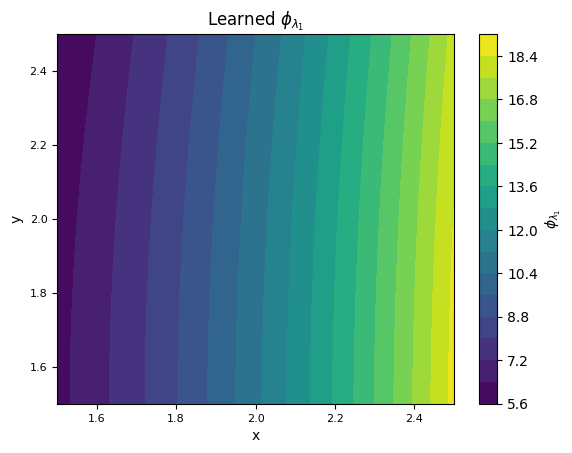}
}
\hfill
\subfloat{\label{fig:bigtest2}
\includegraphics[width=.31\textwidth]{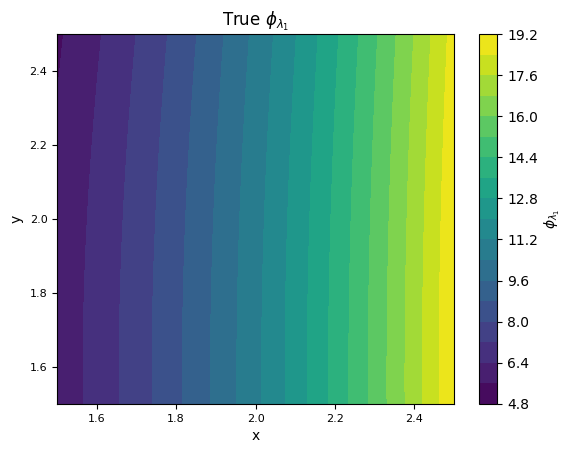}
}
\hfill
\subfloat{\label{fig:bigerror1}
\includegraphics[width=.31\textwidth]{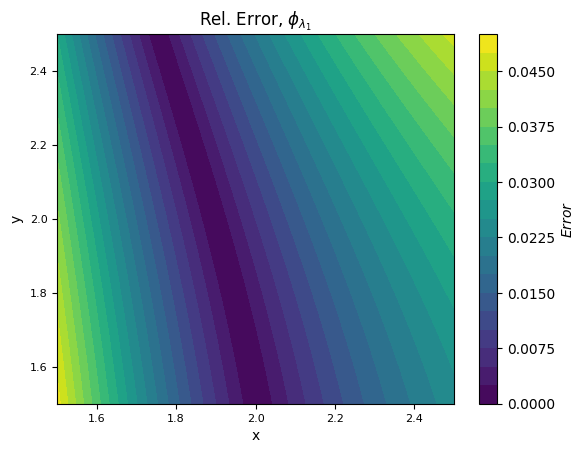}
}
\hfill
\caption{Learned $\phi^*_{\lambda_1}$ (left), true $\phi_{\lambda_1}$ (center) and relative error (right)}
\label{linplot1a}
\end{figure}

\begin{figure}[tbh]
\centering
\subfloat{\label{fig:bigtest3}
\includegraphics[width=.31\textwidth]{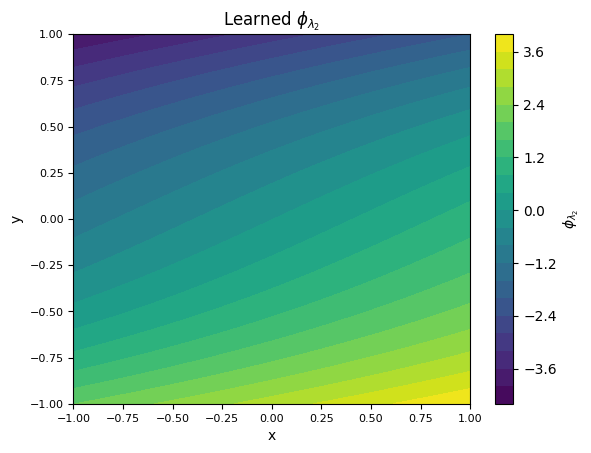}
}
\hfill
\subfloat{\label{fig:bigtest4}
\includegraphics[width=.31\textwidth]{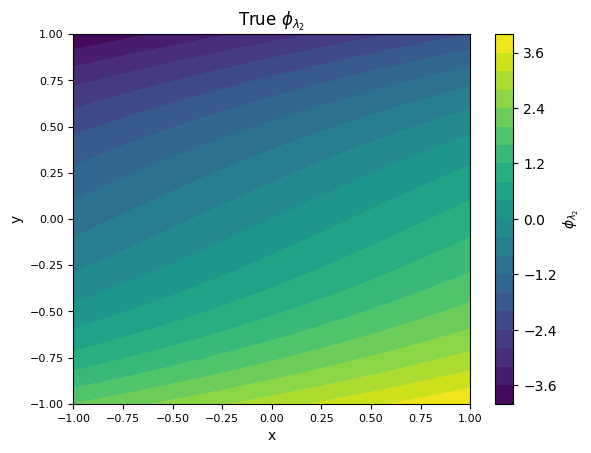}
}
\hfill
\subfloat{\label{fig:bigerror2}
\includegraphics[width=.31\textwidth]{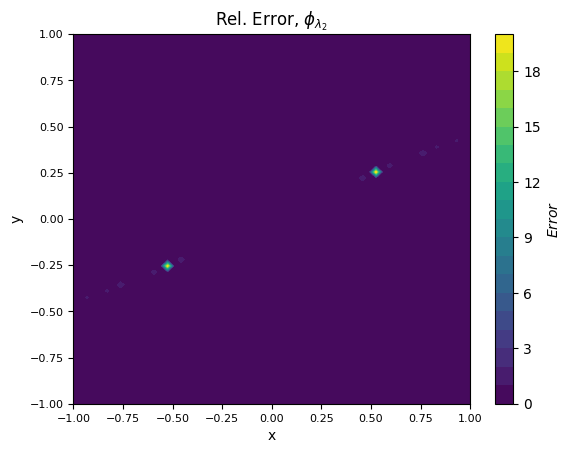}
}
\caption{Learned $\phi^*_{\lambda_2}$ (left), true $\phi_{\lambda_2}$ (center) and relative error (right)}
\label{linplot2a}
\end{figure}


 \subsection{The Duffing Oscillator}


Consider the unforced Duffing oscillator, described by 
\begin{align}
\dot x_1&=x_2\nonumber\\
\dot x_2&=-\delta x_2-x_1(\beta+\alpha x_1^2)
\end{align}
with $\delta=0.5$, $\beta=-1$, and $\alpha=1$,
where $z\in \mathbb{R}$ and $\dot{z}\in \mathbb{R}$ are the scalar position and velocity, respectively. 
The dynamics admit two stable equilibrium points at $(-1,0)$ and $(1,0)$, and one unstable equilibrium point at the origin. In this example, we sample $2500$ points over the domain $[-2,2] \times [-2,2]$ and use the 2D Gaussian kernel with $\sigma_1 = \sigma_2 = 15$. Our results in Figure \ref{duffing} depict $\phi^*_{\lambda_1}$ for $\lambda_1 = \frac{-1+\sqrt{17}}{4}$; this time, we consider only the larger eigenvalue. The plot accurately captures the behavior of the eigenfunction of the Koopman operator corresponding to $\lambda_1$ up to scaling.
 
\begin{figure}[tbh]
\centering
\subfloat{\label{duffing1_new}
\includegraphics[width=0.45\textwidth]{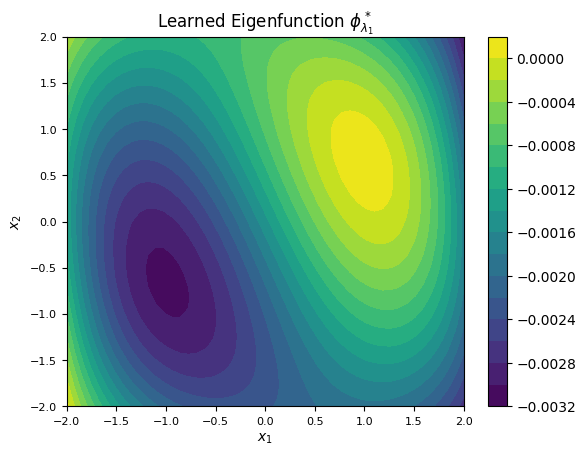}
}
\hfill
 \caption{Learned $\phi^*_{\lambda_1}$}
\label{duffing}
\end{figure}

\subsection{A Three Dimensional Gradient System}
Consider a three-dimensional gradient system of the form
\begin{align}
    \dot x=-\frac{\partial V}{\partial x}, \quad
    x=(x_1\;x_2\;x_3)^\top, 
\end{align}
where the potential function $V$ is given by
\begin{align}
  V(x)=x^\top Px+e^{-(x_1-x_2)^2}, 
\end{align}
where $P= \begin{pmatrix}
0.2&0.1&0.05\\  
0.1&0.3&0.05\\
0.05&0.05&0.2
\end{pmatrix}$ is a positive definite matrix. The system admits an unstable equilibrium at the origin and two stable minima at $ x = (0.90, -0.73, 0.14)$ and $ x = (-0.90,  0.73, 0.14)$. We first compute Jacobian matrix of $-\frac{\partial V}{\partial x}$ at $x=0$ which has three real eigenvalues, i.e., $3.70$, $-0.29$, and $-0.81$.
To construct $\phi_\lambda(x)$, we solve \eqref{optproblem} over the domain $[-2,2] \times [-2,2] \times [-2,2]$ by sampling $3379$ points and use the Gaussian kernel 
$K(x_1,x_2) = \exp\left(-\frac{(x_1-x_2)^2}{2 \times 1.1 ^2} \right)$.
To visualize the result, we plot the level curve of $\phi_\lambda(x)$ with the unstable eigenvalue $\lambda = 3.70$ with a fixed $z=0.57$. As shown in Figure \ref{3D_gradient}, the level curve reveals two regions of attraction centered at the local minima of $V(x)$, i.e.,  $x_1,x_2$.


\begin{figure}[tbh]
\centering
\subfloat{
\includegraphics[width=0.45\textwidth]{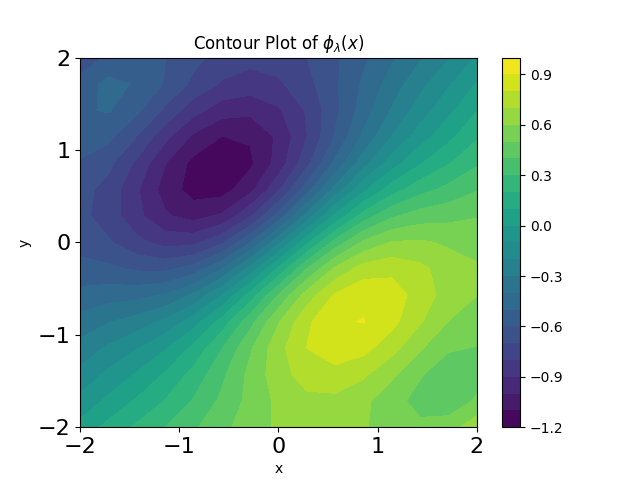}
}
\hfill
 \caption{Learned $\phi^*_{\lambda_1}$}
\label{3D_gradient}
\end{figure}


\section{Code}
The code for this paper can be found at https://github.com/jonghyeon1998/Koopman



\section{Conclusion}
In this paper, we have introduced a novel kernel-based method for approximating the principal eigenfunctions of the Koopman operator, offering a computationally efficient alternative to directly calculating the operator itself. By leveraging the decomposition of the eigenfunctions into linear and nonlinear components, we have shown that the linear part corresponds to the local linear dynamics of the system, while the nonlinear part can be obtained through kernel methods.  

Our approach not only makes the computation of Koopman eigenfunctions more tractable but also enhances the ability to analyze complex dynamical systems by providing valuable insights into their long-term behavior. Overall, the kernel-based method for the construction of eigenfunctions of the Koopman operator presented here represents a significant step toward bridging the gap between linear and nonlinear dynamics within a rigorous mathematical framework.

\section*{Acknowledgement}
GS is a member of INdAM-GNCS, and his work was partially supported by the project ``Data-driven discovery and control of multi-scale interacting artificial agent system'' funded by the program Next-GenerationEU - National Recovery and Resilience Plan (NRRP) -- CUP H53D23008920001. HO and JL acknowledge support from the Air Force Office of Scientific Research under MURI award number FA9550-20-1-0358 (Machine Learning and Physics-Based Modeling and Simulation) and the Department of Energy under the MMICCs SEA-CROGS award.  BH acknowledges support from NSF EPCN-2031570, the Air Force Office of Scientific Research (award number  FA9550-21-1-0317) and the Department of Energy (award number SA22-0052-S001). HO is grateful for support from a Department of Defense Vannevar Bush Faculty
Fellowship.
\bibliographystyle{IEEEtran}
\bibliography{ref, missing_dyn, references_p4}

\def\cprime{$'$} \def\cprime{$'$} \def\cprime{$'$} \def\cprime{$'$}
  \def\cprime{$'$}
\begin{thebibliography}{10}
\providecommand{\url}[1]{#1}
\csname url@samestyle\endcsname
\providecommand{\newblock}{\relax}
\providecommand{\bibinfo}[2]{#2}
\providecommand{\BIBentrySTDinterwordspacing}{\spaceskip=0pt\relax}
\providecommand{\BIBentryALTinterwordstretchfactor}{4}
\providecommand{\BIBentryALTinterwordspacing}{\spaceskip=\fontdimen2\font plus
\BIBentryALTinterwordstretchfactor\fontdimen3\font minus
  \fontdimen4\font\relax}
\providecommand{\BIBforeignlanguage}[2]{{%
\expandafter\ifx\csname l@#1\endcsname\relax
\typeout{** WARNING: IEEEtran.bst: No hyphenation pattern has been}%
\typeout{** loaded for the language `#1'. Using the pattern for}%
\typeout{** the default language instead.}%
\else
\language=\csname l@#1\endcsname
\fi
#2}}
\providecommand{\BIBdecl}{\relax}
\BIBdecl

\bibitem{kantz97}
H.~Kantz and T.~Schreiber, \emph{Nonlinear Time Series Analysis}.\hskip 1em
  plus 0.5em minus 0.4em\relax USA: Cambridge University Press, 1997.

\bibitem{CASDAGLI1989}
\BIBentryALTinterwordspacing
M.~Casdagli, ``Nonlinear prediction of chaotic time series,'' \emph{Physica D:
  Nonlinear Phenomena}, vol.~35, no.~3, pp. 335 -- 356, 1989. [Online].
  Available:
  \url{http://www.sciencedirect.com/science/article/pii/0167278989900742}
\BIBentrySTDinterwordspacing

\bibitem{yk1}
\BIBentryALTinterwordspacing
J.~L. Hudson, M.~Kube, R.~A. Adomaitis, I.~G. Kevrekidis, A.~S. Lapedes, and
  R.~Farber, ``Nonlinear signal processing and system identification:
  {A}pplications to time series from electrochemical reactions,''
  \emph{Chemical Engineering Science}, vol.~45, no.~8, pp. 2075--2081, 1990.
  [Online]. Available:
  \url{https://www.sciencedirect.com/science/article/pii/000925099080079T}
\BIBentrySTDinterwordspacing

\bibitem{yk2}
R.~Rico-Martinez, K.~Krischer, I.~Kevrekidis, M.~Kube, and J.~Hudson,
  ``Discrete-vs. continuous-time nonlinear signal processing of cu
  electrodissolution data,'' \emph{Chemical Engineering Communications}, vol.
  118, no.~1, pp. 25--48, 1992.

\bibitem{yk3}
O.~Grandstrand, ``Nonlinear system identification using neural networks:
  dynamics and instabilities,'' in \emph{Neural Networks for Chemical
  Engineers}, A.~B. Bulsari, Ed.\hskip 1em plus 0.5em minus 0.4em\relax
  Elsevier: Elsevier, 1995, ch.~16, pp. 409--442.

\bibitem{yk4}
\BIBentryALTinterwordspacing
R.~González-García, R.~Rico-Martínez, and I.~Kevrekidis, ``Identification of
  distributed parameter systems: A neural net based approach,'' \emph{Computers
  and Chemical Engineering}, vol.~22, pp. S965--S968, 1998, european Symposium
  on Computer Aided Process Engineering-8. [Online]. Available:
  \url{https://www.sciencedirect.com/science/article/pii/S0098135498001914}
\BIBentrySTDinterwordspacing

\bibitem{survey_kf_ann}
\BIBentryALTinterwordspacing
A.~Chattopadhyay, P.~Hassanzadeh, K.~V. Palem, and D.~Subramanian,
  ``Data-driven prediction of a multi-scale lorenz 96 chaotic system using a
  hierarchy of deep learning methods: Reservoir computing, ann, and
  {RNN-LSTM},'' \emph{CoRR}, vol. abs/1906.08829, 2019. [Online]. Available:
  \url{http://arxiv.org/abs/1906.08829}
\BIBentrySTDinterwordspacing

\bibitem{Sindy}
\BIBentryALTinterwordspacing
S.~L. Brunton, J.~L. Proctor, and J.~N. Kutz, ``Discovering governing equations
  from data by sparse identification of nonlinear dynamical systems,''
  \emph{Proceedings of the National Academy of Sciences}, vol. 113, no.~15, pp.
  3932--3937, 2016. [Online]. Available:
  \url{https://www.pnas.org/content/113/15/3932}
\BIBentrySTDinterwordspacing

\bibitem{jaideep1}
\BIBentryALTinterwordspacing
J.~Pathak, Z.~Lu, B.~R. Hunt, M.~Girvan, and E.~Ott, ``Using machine learning
  to replicate chaotic attractors and calculate lyapunov exponents from data,''
  \emph{Chaos: An Interdisciplinary Journal of Nonlinear Science}, vol.~27,
  no.~12, p. 121102, 2017. [Online]. Available:
  \url{https://doi.org/10.1063/1.5010300}
\BIBentrySTDinterwordspacing

\bibitem{nielsen2019practical}
A.~Nielsen, \emph{Practical Time Series Analysis: Prediction with Statistics
  and Machine Learning}.\hskip 1em plus 0.5em minus 0.4em\relax O'Reilly Media,
  2019.

\bibitem{abarbanel2012analysis}
H.~Abarbanel, \emph{Analysis of Observed Chaotic Data}, ser. Institute for
  Nonlinear Science.\hskip 1em plus 0.5em minus 0.4em\relax Springer New York,
  2012.

\bibitem{pillonetto2011new}
G.~Pillonetto, M.~H. Quang, and A.~Chiuso, ``A {{New Kernel-Based Approach}}
  for {{NonlinearSystem Identification}},'' \emph{IEEE Transactions on
  Automatic Control}, vol.~56, no.~12, pp. 2825--2840, Dec. 2011.

\bibitem{wang2011predicting}
W.-X. Wang, R.~Yang, Y.-C. Lai, V.~Kovanis, and C.~Grebogi, ``Predicting
  {{Catastrophes}} in {{Nonlinear Dynamical Systems}} by {{Compressive
  Sensing}},'' \emph{Physical Review Letters}, vol. 106, no.~15, p. 154101,
  Apr. 2011.

\bibitem{brunton2016discovering}
S.~L. Brunton, J.~L. Proctor, and J.~N. Kutz, ``Discovering governing equations
  from data by sparse identification of nonlinear dynamical systems,''
  \emph{Proceedings of the National Academy of Sciences}, vol. 113, no.~15, pp.
  3932--3937, 2016.

\bibitem{lusch2018deep}
B.~Lusch, J.~N. Kutz, and S.~L. Brunton, ``Deep learning for universal linear
  embeddings of nonlinear dynamics,'' \emph{Nature Communications}, vol.~9,
  no.~1, p. 4950, Dec. 2018.

\bibitem{callaham2021learning}
J.~L. Callaham, J.~V. Koch, B.~W. Brunton, J.~N. Kutz, and S.~L. Brunton,
  ``Learning dominant physical processes with data-driven balance models,''
  \emph{Nature Communications}, vol.~12, no.~1, p. 1016, Dec. 2021.

\bibitem{kaptanoglu2021physicsconstrained}
A.~A. Kaptanoglu, K.~D. Morgan, C.~J. Hansen, and S.~L. Brunton,
  ``Physics-constrained, low-dimensional models for magnetohydrodynamics:
  {{First-principles}} and data-driven approaches,'' \emph{Physical Review E},
  vol. 104, no.~1, p. 015206, Jul. 2021.

\bibitem{kutz2022parsimony}
J.~N. Kutz and S.~L. Brunton, ``Parsimony as the ultimate regularizer for
  physics-informed machine learning,'' \emph{Nonlinear Dynamics}, Jan. 2022.

\bibitem{koopman1932dynamical}
B.~O. Koopman and J.~v. Neumann, ``Dynamical systems of continuous spectra,''
  \emph{Proceedings of the National Academy of Sciences}, vol.~18, no.~3, pp.
  255--263, 1932.

\bibitem{mezic2005spectral}
I.~Mezi{\'c}, ``Spectral properties of dynamical systems, model reduction and
  decompositions,'' \emph{Nonlinear Dynamics}, vol.~41, no.~1, pp. 309--325,
  2005.

\bibitem{mauroy2016global}
A.~Mauroy and I.~Mezi{\'c}, ``Global stability analysis using the
  eigenfunctions of the koopman operator,'' \emph{IEEE Transactions on
  Automatic Control}, vol.~61, no.~11, pp. 3356--3369, 2016.

\bibitem{umathe2023spectral}
B.~Umathe and U.~Vaidya, ``Spectral koopman method for identifying stability
  boundary,'' \emph{IEEE Control Systems Letters}, 2023.

\bibitem{matavalam2024data}
A.~R. Matavalam, B.~Hou, H.~Choi, S.~Bose, and U.~Vaidya, ``Data-driven
  transient stability analysis using the koopman operator,''
  \emph{International Journal of Electrical Power \& Energy Systems}, vol. 162,
  p. 110307, 2024.

\bibitem{vaidya2022spectral}
U.~Vaidya, ``Spectral analysis of koopman operator and nonlinear optimal
  control,'' in \emph{2022 IEEE 61st Conference on Decision and Control
  (CDC)}.\hskip 1em plus 0.5em minus 0.4em\relax IEEE, 2022, pp. 3853--3858.

\bibitem{VaidyaKHJ2024}
------, ``When {K}oopman meets {H}amilton and {J}acobi,'' \emph{In preprint},
  2024.

\bibitem{deka2023path}
S.~A. Deka, S.~S. Narayanan, and U.~Vaidya, ``Path-integral formula for
  computing koopman eigenfunctions,'' in \emph{2023 62nd IEEE Conference on
  Decision and Control (CDC)}.\hskip 1em plus 0.5em minus 0.4em\relax IEEE,
  2023, pp. 6641--6646.

\bibitem{colbrook2020foundations}
M.~Colbrook, ``The foundations of infinite-dimensional spectral computations,''
  Ph.D. dissertation, University of {C}ambridge, 2020.

\bibitem{williamsEDMD}
M.~O. Williams, I.~G. Kevrekidis, and C.~W. Rowley, ``A data--driven
  approximation of the koopman operator: Extending dynamic mode
  decomposition,'' \emph{Journal of Nonlinear Science}, vol.~25, no.~6, pp.
  1307--1346, 2015.

\bibitem{klus2020data}
S.~Klus, F.~N{\"u}ske, S.~Peitz, J.-H. Niemann, C.~Clementi, and
  C.~Sch{\"u}tte, ``Data-driven approximation of the {K}oopman generator: Model
  reduction, system identification, and control,'' \emph{Physica D: Nonlinear
  Phenomena}, vol. 406, p. 132416, 2020.

\bibitem{peherstorfer2020koopman}
B.~Peherstorfer, S.~L. Brunton, and J.~N. Kutz, ``Data-driven koopman operators
  for model reduction and control of systems with symmetries,'' \emph{SIAM
  Journal on Applied Dynamical Systems}, vol.~19, no.~3, pp. 1894--1920, 2020.

\bibitem{klus2015kernel}
S.~Klus, F.~N{\"u}ske, P.~Koltai, I.~G. Kevrekidis, C.~Sch{\"u}tte, and
  C.~Clementi, ``Kernel-based dynamic mode decomposition,'' \emph{Journal of
  Computational Dynamics}, vol.~2, no.~2, pp. 247--265, 2015.

\bibitem{DAS2020573}
\BIBentryALTinterwordspacing
S.~Das and D.~Giannakis, ``Koopman spectra in reproducing kernel hilbert
  spaces,'' \emph{Applied and Computational Harmonic Analysis}, vol.~49, no.~2,
  pp. 573--607, 2020. [Online]. Available:
  \url{https://www.sciencedirect.com/science/article/pii/S1063520320300427}
\BIBentrySTDinterwordspacing

\bibitem{hou2023sparse}
B.~Hou, S.~Sanjari, N.~Dahlin, S.~Bose, and U.~Vaidya, ``Sparse learning of
  dynamical systems in {RKHS}: An operator-theoretic approach,'' in
  \emph{International Conference on Machine Learning}.\hskip 1em plus 0.5em
  minus 0.4em\relax PMLR, 2023, pp. 13\,325--13\,352.

\bibitem{ishikawa2024koopmanoperatorsintrinsicobservables}
\BIBentryALTinterwordspacing
I.~Ishikawa, Y.~Hashimoto, M.~Ikeda, and Y.~Kawahara, ``Koopman operators with
  intrinsic observables in rigged reproducing kernel hilbert spaces,'' 2024.
  [Online]. Available: \url{https://arxiv.org/abs/2403.02524}
\BIBentrySTDinterwordspacing

\bibitem{CuckerandSmale}
F.~Cucker and S.~Smale, ``On the mathematical foundations of learning,''
  \emph{Bulletin of the American mathematical society}, vol.~39, no.~1, pp.
  1--49, 2002.

\bibitem{chen2021solving}
Y.~Chen, B.~Hosseini, H.~Owhadi, and A.~M. Stuart, ``Solving and learning
  nonlinear pdes with gaussian processes,'' \emph{Journal of Computational
  Physics}, vol. 447, 2021.

\bibitem{batlle2025error}
P.~Batlle, Y.~Chen, B.~Hosseini, H.~Owhadi, and A.~M. Stuart, ``Error analysis
  of kernel/gp methods for nonlinear and parametric pdes,'' \emph{Journal of
  Computational Physics}, vol. 520, p. 113488, 2025.

\bibitem{houman_cgc}
\BIBentryALTinterwordspacing
H.~Owhadi, ``Computational {G}raph {C}ompletion,'' \emph{Research in the
  Mathematical Sciences}, vol. 9(2), no.~27, 2021. [Online]. Available:
  \url{https://arxiv.org/abs/2110.10323}
\BIBentrySTDinterwordspacing

\bibitem{bourdais2024codiscovering}
T.~Bourdais, P.~Batlle, X.~Yang, R.~Baptista, N.~Rouquette, and H.~Owhadi,
  ``Codiscovering graphical structure and functional relationships within data:
  A gaussian process framework for connecting the dots,'' \emph{Proceedings of
  the National Academy of Sciences}, vol. 121, no.~32, p. e2403449121, 2024.

\bibitem{bhcm11}
B.~Haasdonk, B.~Hamzi, G.~Santin, and D.~Wittwar, ``Greedy kernel methods for
  center manifold approximation,'' \emph{Proc. of ICOSAHOM 2018, International
  Conference on Spectral and High Order Methods}, vol. 427, no.~1, 2018,
  \url{https://arxiv.org/abs/1810.11329}.

\bibitem{bhcm1}
------, ``Kernel methods for center manifold approximation and a weak
  data-based version of the center manifold theorems,'' \emph{Physica D}, 2021.

\bibitem{lyap_bh}
P.~Giesl, B.~Hamzi, M.~Rasmussen, and K.~Webster, ``Approximation of {L}yapunov
  functions from noisy data,'' \emph{Journal of Computational Dynamics},
  vol.~7, pp. 57--81, 2019, \url{https://arxiv.org/abs/1601.01568}.

\bibitem{BHPhysicaD}
\BIBentryALTinterwordspacing
B.~Hamzi and H.~Owhadi, ``Learning dynamical systems from data: A simple
  cross-validation perspective, {P}art {I}: Parametric {K}ernel {F}lows,''
  \emph{Physica D: Nonlinear Phenomena}, vol. 421, p. 132817, 2021. [Online].
  Available:
  \url{https://www.sciencedirect.com/science/article/pii/S0167278920308186}
\BIBentrySTDinterwordspacing

\bibitem{hamzi2019kernel}
B.~Hamzi and F.~Colonius, ``Kernel methods for the approximation of
  discrete-time linear autonomous and control systems,'' \emph{SN Applied
  Sciences}, vol.~1, no.~7, pp. 1--12, 2019.

\bibitem{bh2020b}
S.~Klus, F.~Nuske, and B.~Hamzi, ``Kernel-based approximation of the {K}oopman
  generator and {S}chr{\"o}dinger operator,'' \emph{Entropy}, vol.~22, 2020,
  \url{https://www.mdpi.com/1099-4300/22/7/722}.

\bibitem{ALEXANDER2020132520}
\BIBentryALTinterwordspacing
R.~Alexander and D.~Giannakis, ``Operator-theoretic framework for forecasting
  nonlinear time series with kernel analog techniques,'' \emph{Physica D:
  Nonlinear Phenomena}, vol. 409, p. 132520, 2020. [Online]. Available:
  \url{http://www.sciencedirect.com/science/article/pii/S016727891930377X}
\BIBentrySTDinterwordspacing

\bibitem{bhks}
A.~B. S. K. B. H.~P. Koltai{,} and C.~Schutte, ``Dimensionality reduction of
  complex metastable systems via kernel embeddings of transition manifold,''
  \emph{Journal of Nonlinear Science}, vol. 31(3), 2019.

\bibitem{bh12}
J.~Bouvrie and B.~Hamzi, ``Empirical estimators for stochastically forced
  nonlinear systems: Observability, controllability and the invariant
  measure,'' \emph{Proc. of the 2012 American Control Conference}, pp.
  294--301, 2012, \url{https://arxiv.org/abs/1204.0563v1}.

\bibitem{bh17}
------, ``Kernel methods for the approximation of nonlinear systems,''
  \emph{SIAM J. Control and Optimization}, 2017,
  \url{https://arxiv.org/abs/1108.2903}.

\bibitem{hb17}
------, ``Kernel methods for the approximation of some key quantities of
  nonlinear systems,'' \emph{Journal of Computational Dynamics}, vol.~1, 2017,
  \url{http://arxiv.org/abs/1204.0563}.

\bibitem{mmd_kernels_bh}
\BIBentryALTinterwordspacing
B.~Hamzi, C.~Kuehn, and S.~Mohamed, ``A note on kernel methods for multiscale
  systems with critical transitions,'' \emph{Mathematical Methods in the
  Applied Sciences}, vol.~42, no.~3, pp. 907--917, 2019. [Online]. Available:
  \url{https://onlinelibrary.wiley.com/doi/abs/10.1002/mma.5394}
\BIBentrySTDinterwordspacing

\bibitem{bh_kfs_p5}
L.~Yang, X.~Sun, B.~Hamzi, H.~Owhadi, and N.~Xie, ``Learning dynamical systems
  from data: A simple cross-validation perspective, {P}art {V}: {S}parse
  {K}ernel {F}lows for 132 chaotic dynamical systems,'' \emph{Physica D}, vol.
  460, 2022.

\bibitem{bh_kfs_p6}
L.~Yang, B.~Hamzi, Y.~Kevrekidis, H.~Owhadi, X.~Sun, and N.~Xie, ``Learning
  dynamical systems from data: A simple cross-validation perspective, {P}art
  {VI}: Hausdorff-metric based {K}ernel {F}lows to learn {A}ttractors and
  {I}nvariant {S}ets,'' \emph{Physica D}, 2023.

\bibitem{hou2024propagating}
B.~Hou, A.~R.~R. Matavalam, S.~Bose, and U.~Vaidya, ``Propagating uncertainty
  through system dynamics in reproducing kernel hilbert space,'' \emph{Physica
  D: Nonlinear Phenomena}, p. 134168, 2024.

\bibitem{cole_hopf_poincare}
\BIBentryALTinterwordspacing
J.~Lee, B.~Hamzi, Y.~Kevrekidis, and H.~Owhadi, ``Gaussian processes simplify
  differential equations,'' September 2024. [Online]. Available:
  \url{https://www.researchgate.net/publication/383680496_Gaussian_Processes_Simplify_Differential_Equations}
\BIBentrySTDinterwordspacing

\bibitem{kernel_sos}
\BIBentryALTinterwordspacing
D.~Lengyel, B.~Hamzi, H.~Owhadi, and P.~Parpas, ``Kernel sum of squares for
  data adapted kernel learning of dynamical systems from data: A global
  optimization approach,'' \emph{arXiv preprint arXiv:2408.06465}, 2024.
  [Online]. Available: \url{https://www.arxiv.org/pdf/2408.06465}
\BIBentrySTDinterwordspacing

\bibitem{santinhaasdonk19}
G.~Santin and B.~Haasdonk, ``Kernel methods for surrogate modeling,''
  \emph{System and Data-Driven Methods and Algorithms}, 2019,
  \url{https://arxiv.org/abs/1907.105566}.

\bibitem{colbrook2024limits}
M.~J. Colbrook, I.~Mezi{\'c}, and A.~Stepanenko, ``Limits and powers of koopman
  learning,'' \emph{arXiv preprint arXiv:2407.06312}, 2024.

\bibitem{bevanda2024koopman}
P.~Bevanda, M.~Beier, A.~Lederer, S.~Sosnowski, E.~H{\"u}llermeier, and
  S.~Hirche, ``Koopman kernel regression,'' \emph{Advances in Neural
  Information Processing Systems}, vol.~36, 2024.

\bibitem{mezic2020spectrum}
I.~Mezi{\'c}, ``Spectrum of the koopman operator, spectral expansions in
  functional spaces, and state-space geometry,'' \emph{Journal of Nonlinear
  Science}, vol.~30, no.~5, pp. 2091--2145, 2020.

\bibitem{mezic2021koopman}
------, ``Koopman operator, geometry, and learning of dynamical systems,''
  \emph{Not. Am. Math. Soc.}, vol.~68, no.~7, pp. 1087--1105, 2021.

\bibitem{Lasota}
A.~Lasota and M.~C. Mackey, \emph{Chaos, Fractals, and Noise: Stochastic
  Aspects of Dynamics}.\hskip 1em plus 0.5em minus 0.4em\relax New York:
  Springer-Verlag, 1994.

\bibitem{budivsic2012applied}
M.~Budi{\v{s}}i{\'c}, R.~Mohr, and I.~Mezi{\'c}, ``Applied koopmanism,''
  \emph{Chaos: An Interdisciplinary Journal of Nonlinear Science}, vol.~22,
  no.~4, p. 047510, 2012.

\bibitem{arnold2012geometrical}
V.~Arnold, \emph{Geometrical Methods in the Theory of Ordinary Differential
  Equations}.\hskip 1em plus 0.5em minus 0.4em\relax Springer Verlag, 2012.

\bibitem{aronszajn50reproducing}
\BIBentryALTinterwordspacing
N.~Aronszajn, ``Theory of reproducing kernels,'' \emph{Transactions of the
  American Mathematical Society}, vol.~68, no.~3, pp. 337--404, 1950. [Online].
  Available: \url{http://dx.doi.org/10.2307/1990404}
\BIBentrySTDinterwordspacing

\bibitem{owhadi_scovel_2019}
H.~Owhadi and C.~Scovel, \emph{Operator-Adapted Wavelets, Fast Solvers and
  Numerical Homogenization: From a Game Theoretic Approach to Numerical
  Approximation and Algorithm Design}.\hskip 1em plus 0.5em minus 0.4em\relax
  Cambridge Monographs on Applied and Computational Mathematics. Cambridge
  University Press, 2019.

\bibitem{Narcowich2004}
F.~J. Narcowich, J.~D. Ward, and H.~Wendland, ``Sobolev bounds on functions
  with scattered zeros, with applications to radial basis function surface
  fitting,'' \emph{Mathematics of Computation}, vol.~74, no. 250, pp. 743--763,
  2005.

\end{thebibliography}
\end{document}